\documentclass[12pt]{amsart}
\usepackage{amssymb,latexsym}
\usepackage{enumerate}

\makeatletter
\@namedef{subjclassname@2010}{%
  \textup{2010} Mathematics Subject Classification}
\makeatother

\newtheorem{thm}{Theorem}[section]
\newtheorem{cor}[thm]{Corollary}

\newtheorem{prop}[thm]{Proposition}

\theoremstyle{definition}

\numberwithin{equation}{section}

\def \varpi {\bar \omega}

\def \R {I \! \! R}

\def \H {I \! \! H\! \! I}

\def \C {\mathbb{C}}
\def\norm {\mid\!\mid}

\begin{document}

%%%%% To ease editing, for IMPAN journals add:

\baselineskip=17pt

%%%%%%%%%%%

%% In the running head, replace first names by initials 
%% and give an abbreviation of the title.

\title[On quaternionic functions]{On quaternionic functions}

\author[P. Dolbeault]{Pierre Dolbeault}
\address{Institut de Math\'ematiques de Jussieu\\
UPMC, 4, place Jussieu 75005 Paris, France}
\email{pierre.dolbeault@upmc.fr}

\date{}

\begin{abstract}
Several sets of quaternionic functions are 
described and studied. Residue current of the right inverse of a quaternionic function is introduced in particular cases.

\end{abstract}



\maketitle

\section{Introduction}

We will work with the definition of quaternions using pairs of complex numbers and with a modified Cauchy-Fueter operator that have been introduced in [CLSSS 07]. 
We will only use right multiplication; the (right) inverse of a nonzero quaternion is defined. We will consider (for simplicity) $C^\infty$ $\H$-valued quaternionic functions defined on an open set $U$ of $\H$ containing $0$. If such a function does not vanish over $U$, it has an (algebraic) inverse which is defined almost everywhere on $U$. Examples are given (section 2).

The origin of this research is a tentative of extension to right inverse of a quaternionic function of the notion of residue current of a meromorphic differential 1-form of one complex variable, which will be developed in section 5. In one complex variable, if the given function is holomorphic, with isolated zeros of finite multiplicity, its inverse is meromorphic, then holomorphic outside the set of poles; so it is natural to search when this property extends to hyperholomorphic functions.                   

In section 3, we characterize the quaternionic functions which are  hyperholomorphic and whose inverses are hyperholomorphic almost everywhere, on $U$, as the solutions of a system of two non linear PDE. We only find non trivial examples of a solution, showing that the considered space of functions is significant: we will call theses functions hypermeromorphic. This also defines a space of germs of functions at 0.

In section 4, we try to describe a subspace ${\mathcal H}_U$ of hyperholomorphic and hypermeromorphic functions defined almost everywhere on $U$, having "good properties for addition and multiplication"; we obtain again systems of non linear PDE, and we give first results on, mainly unknown, spaces of functions.

In section 5, we first recall Cauchy principal value and residue current in $\C$, locally at 0. Afterwards, we define and study, locally, Cauchy principal value and residue current for the inverse of a quaternionic function, in very particular cases, and in relation with the classical theory in two complex variables. 

This paper is a first announcement of a more complete one in progress.

\section{Quaternions. $\H$-valued functions. [CLSSS 07]}

\subsection {Quaternions} If $q\in\H$, then $q=z_1+z_2{\bf j}$ where $z_1, z_2\in\C$, 
hence $\H\cong\C^2\cong\R^4$ as complex or real vector space. We have:
$z_1{\bf j}={\bf j}\overline z_1$ (by computation in real coordinates); by definition, the {\it modulus} of $q$ is $\norm q\norm=({\vert z_1\vert} ^2+{\vert z_2\vert} ^2)^{\frac{1}{2}}$.

The {\it conjugate} of $q$ is $\overline q= \overline z_1-z_2{\bf j}$. Let
* denote the {\it (right) multiplication} in $\H$: 

$q*\overline q=(z_1+z_2{\bf
j})*(\overline z_1-z_2{\bf j})=\vert z_1\vert ^2-z_1z_2{\bf j}+z_2{\bf
j}\overline z_1-z_2{\bf j}z_2{\bf j}=\vert z_1\vert ^2+\vert z_2\vert ^2$,
then: the {\it (right) inverse} of $q=z_1+z_2{\bf j}$ is: $(\vert z_1\vert ^2+\vert
z_2\vert ^2)^{-1}\overline q=(\vert z_1\vert ^2+\vert
z_2\vert ^2)^{-1}(\overline z_1-z_2{\bf j})$. Moreover: $(\vert z_1\vert ^2+\vert
z_2\vert ^2)^{-1}(\overline z_1-z_2{\bf j})*(z_1+z_2{\bf j})=1$,  so the right inverse of $q^{-1}$ is $q$. 

\subsection {Quaternionic functions.}

Let $U$ be an open set of $\H\cong\C^2$ and $f\in C^\infty (U,\H)$, then 
$f=f_1+f_2{\bf j}$, where $f_1, f_2\in C^\infty (U,\C)$. The complex valued functions $f_1,f_2$ will be called the {\it components} of $f$.\vskip 1mm
 
 Remark that $\H$ is a real vector space in which real analysis is valid, in particular differential forms, distributions and currents are defined in $\H$. 
 \vskip 1mm
 Remark that $\displaystyle\frac{\partial
f_1}{\partial\overline z_1}{\bf j}={\bf j}\frac{\partial\overline 
f_1}{\partial z_1}$ and analogous relations.

\subsection {Modified Cauchy-Fueter operator ${\mathcal D}$. 
Hyperholomorphic functions. ([CLSSS 07], [F 39])} 
For $f\in C^\infty (U,\H)$, with $f=f_1+f_2{\bf j}$, where $f_1, f_2\in C^\infty (U,\C)$, 

 $\displaystyle {\mathcal D} f(q)=\frac{1}{2}\big(\frac{\partial}{\partial\overline z_1}
+{\bf j}\frac{\partial}{\partial\overline z_2}\big) f(q)=\frac{1}{2}\big({\frac{\partial
f_1}{\partial\overline z_1}-\frac{\partial\overline f_2}{\partial z_2}}\big)
(q)+{\bf j}\frac{1}{2}\big({\frac{\partial f_1}{\partial\overline
z_2}+\frac{\partial\overline f_2}{\partial z_1}}\big) (q)$.

A function $f\in C^\infty (U,\H)$ s said to be {\it hyperholomorphic} if ${\mathcal
D}f=0$.

Characterization of the hyperholomorphic function $f$ on $U$:
${\displaystyle\frac{\partial
f_1}{\partial\overline z_1}-\frac{\partial\overline f_2}{\partial z_2}}=0; {\displaystyle\frac{\partial f_1}{\partial\overline
z_2}+\frac{\partial\overline f_2}{\partial z_1}}=0$, on $U$.

The conditions:$f_1$ is holomorphic and: $f_2$ is holomorphic are equivalent.
So {\it holomorphic functions} will be identified with  hyperholomorphic functions $f$ such that $f_2=0$. 

Let $f'=f'_1+f'_2{\bf j}$, $f"=f"_1+f"_2{\bf j}$ be two hyperholomorphic functions. 

For every $\alpha\in \H$, ${\mathcal D} (f'\alpha)=0$, ${\mathcal D} (f'+f")={\mathcal D}f'+{\mathcal D} f"=0$.\vskip 1mm

\begin{prop}\label{Proposition 2.3.1.}The set ${\mathcal H}$ of almost everywhere defined hyperholomorphic functions is an $\H$-right vector space.
\end{prop}
 \begin{prop}\label{Proposition 2.3.2.}Let $f'$, $f"$ be two hyperholomorphic functions. Then, their product $f'*f"$ satisfies:

$${\mathcal D} (f'*f")={\mathcal D}f'*{\bf j}f"+\big(f'(\frac{\partial }{\partial\overline z_1})+\overline f'{\bf j}\frac{\partial }{\partial\overline z_2}\big)f"$$
\end{prop}

 \begin{proof}
$f'=f'_1+f'_2{\bf j}$, $f"=f"_1+f"_2{\bf j}$ be two hyperholomorphic functions. 
\vskip 2mm

 We have: $f'*f"=(f'_1+f'_2{\bf j})(f"_1+f"_2{\bf j})
= f'_1f"_1- f'_2\overline f"_2+(f'_1f"_2+f'_2\overline f"_1){\bf j}$
\vskip 1mm

\noindent  Compute
$$\frac{1}{2}\big(\frac{\partial}{\partial\overline z_1}
+{\bf j}\frac{\partial}{\partial\overline z_2}\big) \big(f'_1f"_1- f'_2\overline f"_2+(f'_1f"_2+f'_2\overline f"_1){\bf j}\big)$$
By derivation of the first factors of the sum $f'*f"$, we get the first term: 
$$\frac{1}{2}\big(\frac{\partial f'_1}{\partial\overline z_1}
+{\bf j}\frac{\partial f'_1}{\partial\overline z_2}\big)(f"_1+f"_2{\bf j})     +\frac{1}{2}\big(\frac{\partial f'_2}{\partial\overline z_1}
+{\bf j}\frac{\partial f'_2}{\partial\overline z_2}\big) {\bf j}{\bf j}(\overline f"_2-\overline f"_1{\bf j}) $$

$$=\frac{1}{2}\big(\frac{\partial f'_1}{\partial\overline z_1}
+{\bf j}\frac{\partial f'_1}{\partial\overline z_2}\big)(f"_1+f"_2{\bf j})     +\frac{1}{2}\big(\frac{\partial f'_2{\bf j}}{\partial\overline z_1}
+{\bf j}\frac{\partial f'_2{\bf j}}{\partial\overline z_2}\big) {\bf j}(f"_2{\bf j}+f"_1)= {\mathcal D}f'*{\bf j}f"$$
By derivation in 
$$\frac{1}{2}\big(\frac{\partial}{\partial\overline z_1}
+{\bf j}\frac{\partial}{\partial\overline z_2}\big) \big(f'_1f"_1+f'_2{\bf j} f"_2{\bf j}+(f'_1f"_2{\bf j}+f'_2{\bf j}f"_1)\big)$$
of the second factors of the sum $f'*f"$, we get the second term (up to factor $\frac{1}{2}$):
$$f'_1\frac{\partial f"_1}{\partial\overline z_1}+\overline f'_1{\bf j}\frac{\partial f"_1}{\partial\overline z_2}+f'_1\frac{\partial f"_2}{\partial\overline z_1}{\bf j}+\overline f'_1{\bf j}\frac{\partial f"_2}{\partial\overline z_2}{\bf j}
 +f'_2{\bf j}\frac{\partial f"_2}{\partial\overline z_1}{\bf j}+\overline f'_2{\bf j}\frac{\partial f"_2}{\partial\overline z_2}+f'_2{\bf j}\frac{\partial f"_1}{\partial\overline z_1}+\overline f'_2{\bf j}{\bf j}\frac{\partial f"_1}{\partial\overline z_2} $$

$$=(f'_1+f'_2{\bf j})(\frac{\partial }{\partial\overline z_1})(f"_1+ f"_2{\bf j})   +(\overline f'_1+\overline f'_2{\bf j}){\bf j}\frac{\partial }{\partial\overline z_2}(f"_1+ f"_2{\bf j})= \big((f'_1+f'_2{\bf j})(\frac{\partial }{\partial\overline z_1})+(\overline f'_1+\overline f'_2{\bf j}){\bf j}\frac{\partial }{\partial\overline z_2}\big)(f"_1+ f"_2{\bf j})$$
$$= \big(f'(\frac{\partial }{\partial\overline z_1})+\overline f'{\bf j}\frac{\partial }{\partial\overline z_2}\big)f" $$
\end{proof}

If the components of $f'$ and $f"$ are real, the second term is: 

$$\frac{1}{2}(f'_1+f'_2{\bf j})(\frac{\partial }{\partial\overline z_1}+{\bf j}\frac{\partial }{\partial\overline z_2})(f"_1+ f"_2{\bf j})=f'*
{\mathcal D}f"$$
i.e.
\noindent\vskip 1mm

\begin{cor}\label{Proposition 2.3.3} {\it The set ${\mathcal H}_{\R}$ of almost everywhere defined hyperholomorphic functions whose components  are  real is an $\R$-right algebra.} 
\end{cor}

\subsection {Null set and inverse of a quaternionic function.}

We call {\it inverse} of a function $f: q\mapsto f(q)$, the function $f^{-1}: q\mapsto f(q)^{-1}$. Let $f=f_1+f_2{\bf j}$ be a quaternionic function on $U$. 
The null set $Z(f)$ satisfies: $f_1=0; f_2=0$, then  $Z(f)$ is of measure 0 in $U$. Ex.: $f_1=\overline z_1$; $f_2=\overline z_2$, then
$Z(f)=\{0\}$. Note that if $f$ is holomorphic, then, $f_2\equiv 0$
and $Z(f)$ is a complex hypersurface in $\C^2$.

{\it Inversion and hyperholomorphy}. The inverse of the quaternionic function $f$ is the peculiar quaternionic function defined almost everywhere on $U$:
$$\frac{1}{f}=(\vert f_1\vert ^2+\vert f_2\vert ^2)^{-1}(\overline f_1-f_2{\bf j})=\vert f\vert^{-1} \overline f$$ 
where $\overline f$ is the (quaternionic) conjugate of $f$.

{\it Assume $f$ to be hyperholomorphic and $Z(f)=\{0\}$, then
$\displaystyle\frac{1}{f}$ is not necessarily hyperholomorphic outside $\{0\}$}. 

Ex.: $f=\overline z_1+\overline z_2{\bf j}$, then
$$\frac{1}{f}=(z_1\overline z_1+z_2\overline z_2)^{-1}(z_1-\overline z_2{\bf j}); \hskip    2mm{\mathcal D}(\frac{1}{f})\not =0, $$

\noindent {\it Example of a function hyperholomorphic outside $0$.}

$H(q)=(z_1\overline z_1+z_2\overline z_2)^{-2}(\overline z_1-\overline z_2{\bf j})$ {\it is hyperholomorphic} since:

\vskip 1mm
\noindent $\displaystyle {\mathcal D}H(q)=\frac{1}{2}(z_1\overline z_1+z_2\overline z_2)^{-3}(-2z_1\overline z_1+z_1\overline z_1+z_2\overline z_2-2z_2\overline z_2+(z_1\overline  z_1+z_2\overline z_2))=0$. \vskip 1mm

\noindent But $F=z_1+\overline z_2{\bf j}$ {\it is not hyperholomorphic}: 
the conjugate of $F$ is $\overline F=\overline z_1-\overline z_2{\bf j}$; \hskip 2mm  
$(z_1+\overline z_2{\bf j})*(\overline z_1-\overline z_2{\bf j})=z_1\overline z_1+z_2\overline z_2$. 
So $F^{-1}=(\overline z_1-\overline z_2{\bf j})(z_1\overline z_1+z_2\overline z_2)^{-1}$, and $\displaystyle H(q)=F^{-1}(z_1\overline z_1+z_2\overline z_2)^{-1}=\frac{F^{-1}}{\vert F\vert}$.\vskip 1mm
 
($H$ {\it is the Cauchy kernel} for the modified Cauchy-Fueter operator ${\mathcal D}$).

\noindent {\it Inverse of a holomorphic function}.

Let $f= f_1+0{\bf j}$ be a hyperholomorphic function. Then 
$f^{-1}=f_1^{-1}+0{\bf j}$ and $f^{-1}$ is hyperholomorphic outside of the complex hypersurface $Z(f)$. Remark that $Z(f)$ is a subvariety of complex dimension 1, then of measure zero, in $U$. 

 We will consider almost everywhere defined hyperholomorphic functions on $U$. Ex.: holomorphic,  meromorphic functions.

\section{Hyperholomorphic functions whose inverses are hyperholomorphic almost everywhere.}

\begin{prop}\label{Proposition 3.2.} The following conditions are equivalent

$(i)$ the function $f=f_1+f_2{\bf j}$ and its right inverse are hyperholomorphic, when they are defined; 

$(ii)$
we have the equations: 

$$(\overline f_1- f_1)\frac{\partial\overline f_1}{\partial z_1}-\overline f_2\frac{\partial f_2}{\partial  z_1}-f_2\frac{\partial\overline  f_1}{\partial \overline z_2}=0,\leqno (3) $$

$$\overline f_2\frac{\partial f_1}{\partial z_1}+ \frac{\partial\overline f_2}{\partial z_1}(\overline f_1-f_1)-f_2\frac{\partial\overline f_2}{\partial\overline z_2}=0. \leqno (4)$$
\end{prop}

\begin{proof}
Let $f=f_1+f_2{\bf j}$ be a hyperholomorphic function and $g=g_1+g_2{\bf j}=\displaystyle \vert f\vert^{-1}(\overline f_1-f_2{\bf j})$ its inverse; so $g_1=\vert f\vert^{-1}\overline f_1$; $g_2=-\vert f\vert^{-1}f_2$, 
 where $\vert f\vert=(f_1\overline f_1+f_2\overline f_2)$.
 
 $\displaystyle{\mathcal D}g(q)=\frac{1}{2}\big(\frac{\partial}{\partial\overline z_1}
+{\bf j}\frac{\partial}{\partial\overline z_2}\big) g(q)
=\frac{1}{2}\big({\frac{\partial
g_1}{\partial\overline z_1}-\frac{\partial\overline g_2}{\partial z_2}}\big)
(q)+{\bf j}\frac{1}{2}\big({\frac{\partial g_1}{\partial\overline
z_2}+\frac{\partial\overline g_2}{\partial z_1}}\big) (q)$

$$\frac{\partial g_1}{\partial\overline z_1}= \vert f\vert^{-1}\frac{\partial\overline  f_1}{\partial \overline z_1}-\vert f\vert^{-2} \overline  f_1\big(\frac{\partial f_1}{\partial \overline z_1}\overline f_1+f_1\frac{\partial\overline  f_1}{\partial\overline  z_1}+\frac{\partial f_2}{\partial \overline z_1}\overline f_2+f_2\frac{\partial\overline  f_2}{\partial \overline z_1}\big), {\rm etc}$$

$$ 2\vert f\vert^2{\mathcal D}g=(f_1\overline f_1+f_2\overline f_2)(\frac{\partial\overline  f_1}{\partial \overline z_1}+\frac{\partial\overline  f_2}{\partial z_2})-\overline  f_1  f_1\frac{\partial \overline f_1}{\partial \overline z_1}-\overline  f_1\overline  f_1\frac{\partial f_1}{\partial \overline z_1}-\overline f_1f_2\frac{\partial\overline  f_2}{\partial \overline z_1}\leqno (**) $$
$$-\overline f_1\overline f_2\frac{\partial  f_2}{\partial \overline z_1}-\overline  f_1\overline  f_2\frac{\partial  f_1}{\partial  z_2}-  f_1\overline  f_2\frac{\partial\overline f_1}{\partial  z_2}.-
\overline f_2\overline f_2\frac{\partial  f_2}{\partial  z_2}- f_2\overline f_2\frac{\partial\overline  f_2}{\partial  z_2}$$
$$+{\bf j}\Big((f_1\overline f_1+f_2\overline f_2)(\frac{\partial\overline  f_1}{\partial \overline z_2}-\frac{\partial\overline  f_2}{\partial z_1})-\overline  f_1\overline   f_1\frac{\partial  f_1}{\partial \overline z_2}-\overline  f_1 f_1\frac{\partial \overline f_1}{\partial \overline z_2}-\overline f_1\overline f_2\frac{\partial   f_2}{\partial \overline z_2}-\overline f_1 f_2\frac{\partial\overline  f_2}{\partial \overline z_2} $$

$$+\overline  f_1\overline  f_2\frac{\partial  f_1}{\partial  z_1}+ f_1\overline  f_2\frac{\partial\overline f_1}{\partial  z_1}+
\overline f_2\overline f_2\frac{\partial  f_2}{\partial  z_1}+ f_2\overline f_2\frac{\partial\overline  f_2}{\partial  z_1}\Big)$$

$f$ being hyperholomorphic, $g$ hyperholomorphic is equivalent to:

$$ +f_1\overline f_1\frac{\partial\overline  f_2}{\partial z_2} +f_2\overline f_2\frac{\partial\overline  f_1}{\partial \overline z_1} -\overline  f_1\overline  f_1\frac{\partial f_1}{\partial \overline z_1}-\overline f_1f_2\frac{\partial\overline  f_2}{\partial \overline z_1}-\overline  f_1\overline  f_2\frac{\partial  f_1}{\partial  z_2}-
\overline f_2\overline f_2\frac{\partial  f_2}{\partial  z_2}\leqno (**) $$
$$+\overline f_2\frac{\partial  f_2}{\partial \overline z_1}(f_1-\overline f_1)=0$$
$$+f_2\overline f_2\frac{\partial\overline  f_1}{\partial \overline z_2}-\overline f_1\overline f_2\frac{\partial   f_2}{\partial \overline z_2}-\overline f_1 f_2\frac{\partial\overline  f_2}{\partial \overline z_2}+\overline   f_1\frac{\partial  f_1}{\partial \overline z_2}(f_1-\overline f_1)
+\overline  f_1\overline  f_2\frac{\partial  f_1}{\partial  z_1}+ f_1\overline  f_2\frac{\partial\overline f_1}{\partial  z_1}$$
$$+
\overline f_2\overline f_2\frac{\partial  f_2}{\partial  z_1}=0$$

After conjugaison of the first equation, and using

$${\displaystyle\frac{\partial
f_1}{\partial\overline z_1}-\frac{\partial\overline f_2}{\partial z_2}}=0; {\displaystyle\frac{\partial f_1}{\partial\overline
z_2}+\frac{\partial\overline f_2}{\partial z_1}}=0, \leqno (1)$$ we get:

$$  +f_2\overline f_2\frac{\partial f_1}{\partial z_1} + f_1 (\overline f_1- f_1)\frac{\partial\overline f_1}{\partial z_1}-f_1\overline f_2\frac{\partial f_2}{\partial z_1} + f_2\frac{\partial\overline f_2}{\partial z_1}(\overline f_1-f_1) -f_1f_2\frac{\partial \overline f_1}{\partial\overline   z_2}-f_2f_2\frac{\partial\overline f_2}{\partial\overline z_2}=0\leqno (**) $$

$$+\overline  f_1\overline  f_2\frac{\partial  f_1}{\partial  z_1}+ (f_1-\overline f_1)\overline  f_2\frac{\partial\overline f_1}{\partial  z_1}+
\overline f_2\overline f_2\frac{\partial  f_2}{\partial  z_1}+\overline   f_1\frac{\partial  f_1}{\partial \overline z_2}(f_1-\overline f_1)+f_2\overline f_2\frac{\partial\overline  f_1}{\partial \overline z_2}.
-\overline f_1 f_2\frac{\partial\overline  f_2}{\partial \overline z_2}=0$$

Assume  $f_1\not =0$, $f_2\not =0$. After multiplication of the first equation by $f_1$ and of the second by $-f_2$, and sum, we get

$$(\overline f_1- f_1)\frac{\partial\overline f_1}{\partial z_1}-\overline f_2\frac{\partial f_2}{\partial  z_1}-f_2\frac{\partial\overline  f_1}{\partial \overline z_2}=0$$

By an analogous process, we get:

$$\overline f_2\frac{\partial f_1}{\partial z_1}+ \frac{\partial\overline f_2}{\partial z_1}(\overline f_1-f_1)-f_2\frac{\partial\overline f_2}{\partial\overline z_2}=0 $$

\end{proof}

 \begin{cor}\label{ Corollary 3.2.1. } {\it If $f$ satisfies the conditions of the Proposition, the same is true for $\alpha f$ with $\alpha\in \R$}.
 \end{cor}
 Let $f=f_1+0{\bf j}$ be an almost everywhere holomorphic function,  then the condition $(ii)$ of Proposition 3.1 is satisfied. 
 
 Now give another example of quaternionic function satisfying the conditions of Proposition 3.1:
 \begin{prop}\label{Proposition 3.4} Let $f=f_1+f_2{\bf j}$, with $f_1=z_1+\overline z_1+z_2+\overline z_2+A$, $f_2= -z_1-\overline z_1+z_2+\overline z_2+ B$, $A, B\in\R$, then: $f$ and $f^{-1}$ outside the zero set of $f$, are hyperholomorphic 
 The null set of $f= f_1+f_2{\bf j}$, for $f_1,f_2$ as above, for $A=B=0$, is:
$$z_1+\overline z_1+z_2+\overline z_2=0; 
-z_1-\overline z_1+z_2+\overline z_2=0$$
i.e, by difference and sum:
$\overline z_1+z_1=0; z_2+\overline z_2=0,$
i.e. $x_1=0; x_2=0$
in $\R^4$.
\end{prop} 
\begin{proof} $f_1=z_1+\overline z_1+z_2+\overline z_2+A$, $A$ constant;   then: 

(3) \hskip 3mm $\displaystyle -\overline f_2\frac{\partial f_2}{\partial  z_1}-f_2\frac{\partial\overline  f_1}{\partial \overline z_2}=0$, i.e. 
$\displaystyle -\overline f_2\frac{\partial f_2}{\partial  z_1}-f_2=0$, 

{\it Try: $f_2$ real.} Then $\displaystyle \frac{\partial f_2}{\partial  z_1}=-1$ and $f_2=-z_1+C(\overline  z_1,z_2,\overline z_2)=-z_1-\overline z_1+C'(\overline  z_1,z_2,\overline z_2)$, with $C'$ real and  
 $\displaystyle \frac{\partial C'}{\partial  z_1}=0$.

From (4),
 $\displaystyle \frac{\partial f_2}{\partial\overline z_2}=\frac{\partial C'}{\partial\overline z_2}=1$, and $C'= z_2+\overline z_2+ C"(\overline z_1, z_2)$, with $C"$ real.  $f_2= -z_1-\overline z_1+z_2+\overline z_2+ C"(\overline z_1, z_2)$, with $C"$ real, and $\displaystyle \frac{\partial C"}{\partial\overline z_2}=0$, $\displaystyle \frac{\partial C"}{\partial  z_1}=0$.

$C"$ being holomorphic in $z_2$ and $\overline z_1$ is a constant B. 

Hence:  $f_2= -z_1-\overline z_1+z_2+\overline z_2+ B$, with $B\in\R$.
  \end{proof}
  
  \section{On the space of hyperalgebraic functions} 
  
  \subsection{Definition} Let $U$ be an open neighborhood of 0 in $\H\cong\C^2$. {\it From now on,we will only consider the quaternionic functions $f=f_1+f_2{\bf j}$ having the following properties}: 
  
 $(i)$ when $f_1$ and $f_2$ are not holomorphic, the set $Z(f_1)\cap Z(f_2)$ is {\it discrete} on $U$;

$(ii)$ for every $q\in Z(f_1)\cap Z(f_2)$, $J_q^\alpha (.)$ denoting the {\it jet of order $\alpha$ at $q$} [M 66], let $\displaystyle m_i=\sup_{\alpha_i} J_q^{\alpha_i}(f_i)=0$; $m_i$, $i=1,2$, is finite. 

{\it Define}: $\displaystyle m_q=\inf_i m_i$.   
  
Remark that, in this paper, the considered peculiar examples of quaternionic functions $f$ satisfy: the set $Z(f_1)\cap Z(f_2)$  is reduced to one point and that $\alpha_i=1$.
  
  Let $f=f_1+f_2{\bf j}$ be a quaternonic function on $U$ and $g=g_1+g_2{\bf j}=\displaystyle \vert f\vert^{-1}(\overline f_1-f_2{\bf j})$ its inverse; so $g_1=\vert f\vert^{-1}\overline f_1$; $g_2=-\vert f\vert^{-1}f_2$, 
 where $\vert f\vert=(f_1\overline f_1+f_2\overline f_2)$.

 The right inverse of $g$ is $h=h_1+h_2{\bf j}$, with $h_1=\vert g\vert^{-1}\overline g_1$; $h_2=-\vert g\vert^{-1}g_2$; $\vert g\vert=\vert f\vert^{-2}(\overline f_1 f_1+ \overline f_2 f_2)=\vert f\vert^{-1} $; then:  $h_1=\vert g\vert^{-1}\overline g_1=f_1$, ... So {\it the right inverse of $g$ is} $f$.
  
  \subsection{Definition} We will call  {\it hypermeromorphic function}, on $U$, any almost everywhere defined hyperholomorphic function whose right inverse is hyperholomorphic almost everywhere. 
  
  Thanks to Definition 4.1, meromorphic functions in one complex varable are hypermeromorphic.

From Proposition 3.3, the set ${\mathcal M}$ of hypermeromorphic functions is not reduced to the space of meromorphic functions in one complex variable. 

Let ${\mathcal M}_0$ be the set of elements $f$ of ${\mathcal M}$ described in Proposition 3.3. ${\mathcal M}_0$ is an $\R$-vector space; it also contains $f^{-1}$ and the products $f*f^{-1}=1$.

\begin{prop}\label{Proposition 4.2.1} {\it Let $f$ , $g$  be two hypermeromorphic functions on $U$, then the following conditions are equivalent:

$(i)$ the product $f*g$ is hypermeremorphic;

$(ii)$ $f$  and $g$ satisfy the system of} PDE:

$$g_1(\frac{\partial
f_1}{\partial\overline z_1}+\frac{\partial\overline f_2}{\partial z_2})+(f_1-\overline f_1)\frac{\partial g_1}{\partial\overline z_1}+\overline f_2\frac{\partial g_1}{\partial z_2}-f_2\frac{\partial\overline g_2}{\partial\overline z_1}=0$$

$$g_1(\frac{\partial
f_1}{\partial\overline z_2}-\frac{\partial\overline f_2}{\partial z_1})+(f_1-\overline f_1)\frac{\partial g_1}{\partial\overline z_2}-\overline f_2\frac{\partial g_1}{\partial z_1}-f_2\frac{\partial\overline g_2}{\partial\overline z_2}=0$$
\end{prop}

\begin{proof}Let $f=f_1+f_2{\bf j}$ and $g=g_1+g_2{\bf j}$ two hypermeromorphic functions and $f*g=f_1g_1-f_2\overline g_2+(f_1g_2-f_2\overline g_1){\bf j}$ their product, then  
 ${\displaystyle\frac{\partial
f_1}{\partial\overline z_1}-\frac{\partial\overline f_2}{\partial z_2}}=0; {\displaystyle\frac{\partial f_1}{\partial\overline
z_2}+\frac{\partial\overline f_2}{\partial z_1}}=0$ and..., on $U$ and, the  conditions for the product to be hyperholomorphic are:

$$\frac{\partial
(f_1g_1-f_2\overline g_2)}{\partial\overline z_1}-\frac{\partial(\overline f_1\overline g_2-\overline f_2g_1)}{\partial z_2}=$$
$$g_1(\frac{\partial
f_1}{\partial\overline z_1}+\frac{\partial\overline f_2}{\partial z_2})+f_1\frac{\partial g_1}{\partial\overline z_1}-\overline f_1\frac{\partial\overline g_2}{\partial z_2}+\overline f_2\frac{\partial g_1}{\partial z_2}-f_2\frac{\partial\overline g_2}{\partial\overline z_1}=0$$
$$\displaystyle\frac{\partial
(f_1g_1-f_2\overline g_2)}{\partial\overline z_2}+\frac{\partial(\overline f_1\overline g_2-\overline f_2g_1)}{\partial z_1}=$$
$$g_1(\frac{\partial
f_1}{\partial\overline z_2}-\frac{\partial\overline f_2}{\partial z_1})+f_1\frac{\partial g_1}{\partial\overline z_2}+\overline f_1\frac{\partial\overline g_2}{\partial z_1}-\overline f_2\frac{\partial g_1}{\partial z_1}-f_2\frac{\partial\overline g_2}{\partial\overline z_2}=0$$
\end{proof}

\begin{cor}Let $f$, $g$  be two hypermeromorphic functions on $U$,whose components are real, then the following conditions are equivalent:

$(i)$ the product $f*g$ is hypermeremorphic;

$(ii)$ $f$  and $g$ satisfy te system of PDE:

$$g_1(\frac{\partial
f_1}{\partial\overline z_1}+\frac{\partial f_2}{\partial z_2})+ f_2\frac{\partial g_1}{\partial z_2}-f_2\frac{\partial g_2}{\partial\overline z_1}=0$$

$$g_1(\frac{ 
f_1}{\partial\overline z_2}-\frac{\partial f_2}{\partial z_1})- f_2\frac{\partial g_1}{\partial z_1}-f_2\frac{\partial g_2}{\partial\overline z_2}=0$$
\end{cor}

\subsection{Definition.}We will call {\it hyperalgebraic} the  hypermeromorphic functions whose sum and product are hypermeromorphic: see section 4.2.

 \begin{prop} {The set ${\mathcal M}$ of hypermeromorphic functions on $U$ is a subalgebra of the algebra of quaternionic functions.}
 \end{prop}
 
 \begin{prop} The set ${\mathcal A}$ of hyperalgebraic functions on $U$ is a "field" with only associativity of the multiplication. \end{prop}

\section{About residue current in quaternionic analysis: particular cases.}
\subsection{Residue current in an open set of $\C$.} [D10a],[D10b]

Let $\omega= g(z)dz$ be a meromorphic       1-form on a small enough 
 open set $0\in U\subset \C$ having 0 as unique pole, with multiplicity $k$:
 
 $$g=\sum_{l=1}^k \frac {a_{-l}}{z^l}+{\rm holomorphic  \ function}$$

Note that $\omega$ is $d$-closed.\vskip 1mm
 
 Let $\psi=\psi_0 d\overline z\in {\mathcal D}^1(U)$ be a 1-test form. In general $g\psi$ is not integrable, but the Cauchy principal value
 
 $$Vp\lbrack\omega\rbrack(\psi) =\lim_{\epsilon\rightarrow 0}\int_{\vert z\vert\geq\epsilon}\omega\wedge\psi \ \  
$$
 exists as a current, and $d Vp\lbrack\omega\rbrack=d" Vp\lbrack\omega\rbrack={\rm Res}\lbrack\omega\rbrack$ is the residue current of 
$\omega$. For any test function $\varphi$ on $U$,
 
 $${\rm Res}\lbrack\omega\rbrack(\varphi)=\lim_{\epsilon\rightarrow 0}\int_{\vert z\vert=\epsilon}\omega
\wedge\varphi$$
 
 Then ${\rm Res}\lbrack\omega\rbrack=2\pi i \ {\rm res}_0(\omega)\delta_0+
dB=\displaystyle\sum_{j=0}^{k-1}b_j\frac{\partial^j}{\partial z^j}\delta_0$ \  where $ {\rm res}_0(\omega)= a_{-1}$
is the Cauchy residue. We remark that $\delta_0$ is the integration current on the subvariety $\{0\}$ of $U$, that
$D=\displaystyle\sum_{j=0}^{k-1}b_j\frac{\partial^j}{\partial z^j}$ with $b_j=\lambda_j a_{-j}$ where the $\lambda_j$
are universal constants.
 
 Conversely, given the subvariety $\{0\}$ and the differential operator $D$, then the meromorphic differential form $\omega$ is equal to $gdz$, up to holomorphic form; hence the residue current ${\rm Res}\lbrack\omega\rbrack=D\delta_0$, can be constructed. 
 
 \subsection{Cauchy principal value of a 
quaternionic 1-form.}
 Let $f$ be a quaternionic function on an open neighborhood $U$ of 0 in $\H$ satisfying the conditions of Definition 4.1, with $m_i=1$, $i=1,2$. 
 
    Let $\displaystyle\omega=\frac{df}{f}= \frac{1}{f}(df_1+df_2{\bf j})$. We want to extend $\omega$ into a current of degree  1; first, consider the part of type (1,0). Let    
$\psi=\psi_1 d\overline z_1\wedge dz_2 \wedge d\overline  z_2+ \psi_2 dz_1\wedge
d\overline z_1\wedge  d\overline z_2$ be a test form. 

Define:

$$Vp[\omega](\psi)=\lim_{\varepsilon\rightarrow 0}\int_{\vert
f\vert\geq\varepsilon}(\vert f_1\vert ^2+\vert f_2\vert ^2)^{-1} (\overline
f_1-f_2{\bf j})(df_1+df_2{\bf j})\wedge(\psi_1 
d\overline z_1\wedge dz_2 \wedge d\overline  z_2+ \psi_2 dz_1\wedge
d\overline z_1\wedge  d\overline z_2)$$

We have to prove the existence of $Vp[\omega]$ at least if $f$ is hyperholomorphic (Check the proof in the classical case where $f$ is holomorphic in one complex variable or in two complex variables where the proof is less easy, but don't need the resolution of singularities).\vskip 1mm

$$(df_1+df_2{\bf j})\wedge(\psi_1 
d\overline z_1\wedge dz_2 \wedge d\overline  z_2+ \psi_2 dz_1\wedge
d\overline z_1\wedge  d\overline z_2)$$

$$=\big(\frac{\partial f_1}{\partial z_1}\psi_1 +\frac{\partial f_1}{\partial z_2}\psi_2 - (\frac{\partial f_2}{\partial\overline z_1}\overline\psi_1 -\frac{\partial f_2}{\partial\overline z_2}\overline\psi_2){\bf j}\big)
dz_1\wedge d\overline z_1\wedge dz_2 \wedge d\overline  z_2$$

Take polar coordinates: $\lambda=\norm q\norm=(\vert z_1\vert ^2+\vert z_2\vert ^2)^{\frac{1}{2}}$ and the spherical coordinates on $\lambda{\bf S}^3$. Let $d\sigma$ be the volume element on ${\bf S}^3$, and $K$ a convenient universal constant, then:

\centerline {$dz_1\wedge d\overline z_1\wedge dz_2 \wedge d\overline  z_2=K\lambda d\lambda\wedge d\sigma$}

$$Vp[\omega](\psi)=\lim_{\varepsilon\rightarrow 0}\int_{\vert
f\vert\geq\varepsilon}(\vert f_1\vert ^2+\vert f_2\vert ^2)^{-1} (\overline
f_1-f_2{\bf j})\big(\frac{\partial f_1}{\partial z_1}\psi_1 +\frac{\partial f_1}{\partial z_2}\psi_2 - (\frac{\partial f_2}{\partial\overline z_1}\overline\psi_1 -\frac{\partial f_2}{\partial\overline z_2}\overline\psi_2){\bf j}\big)K\lambda d\lambda\wedge d\sigma$$

\vskip 1mm
Same result for the part of type ((0,1) of $Vp[\omega]$.

We will prove the existence of $Vp[\omega]$ in a particular case:\vskip 1mm

 5.2.2. {\it Particular case}: $f_1=\overline z_1$; $f_2=\overline z_2$. Then: ${\mathcal D}f=0$.

$$Vp[\omega](\psi)=\lim_{\varepsilon\rightarrow 0}\int_{\vert
f\vert\geq\varepsilon}(\vert \overline z_1\vert ^2+\vert \overline z_2\vert ^2)^{-1}(\overline
z_1 -z_2{\bf j})(d\overline
z_1+d\overline 
z_2{\bf j})\wedge(\psi_1 
d\overline z_1\wedge dz_2 \wedge d\overline  z_2+ \psi_2 dz_1\wedge
d\overline z_1\wedge  d\overline z_2)$$

$$Vp[\omega](\psi)=\lim_{\varepsilon\rightarrow 0}\int_{\vert
f\vert\geq\varepsilon}(\vert \overline z_1\vert ^2+\vert \overline z_2\vert ^2)^{-1}(\overline z_1 -z_2{\bf j})(\overline \psi_2 dz_1\wedge d\overline z_1\wedge  dz_2\wedge d\overline z_2{\bf j})$$

$$Vp[\omega](\psi)=\lim_{\varepsilon\rightarrow 0}\int_{\lambda\geq\varepsilon}\lambda^{-2}(\overline z_1 -z_2{\bf j})(K\overline \psi_2\lambda d\lambda\wedge d\sigma{\bf j}
)$$

Same result for the part of type ((0,1) of $Vp[\omega]$.

This defines a current of order 0 on $\H$.

\noindent 5.2.3. If $f_2=0$ and $f=f_1$ is {\it holomorphic}, then

$$Vp[\omega](\psi)=\lim_{\varepsilon\rightarrow 0}\int_{\vert
f\vert\geq\varepsilon}\frac{df}{f}\wedge(\psi_1 
d\overline z_1\wedge dz_2 \wedge d\overline  z_2+ \psi_2 dz_1\wedge
d\overline z_1\wedge  d\overline z_2)$$

$$=Vp[\omega](\psi)=\lim_{\varepsilon\rightarrow 0}\int_{\vert
f\vert\geq\varepsilon}\frac{1}{f}(\frac{\partial f}{\partial z_1}\psi_1 +\frac{\partial f}{\partial z_2}\psi_2) dz_1\wedge
d\overline z_1\wedge dz_2 \wedge d\overline  z_2$$

\subsection{Residue.}

\subsubsection{Assume: $Z(f)=\{0\}$.} We want to define a current, Res$[\omega]$, and first its part of type $(1,1)$, on a test form $\varphi=\varphi_{11} dz_1\wedge d\overline z_1+\varphi_{12} dz_1\wedge d\overline z_2+\varphi_{21} dz_2\wedge d\overline z_1+\varphi_{22} dz_2\wedge d\overline z_2$, as follows:

$${\rm Res}[\omega](\varphi)=\lim_{\varepsilon\rightarrow 0}\int_{\vert
f\vert=\varepsilon}(\vert f_1\vert ^2+\vert f_2\vert ^2)^{-1}(\overline f_1-f_2{\bf j})(df_1 +df_2{\bf j})(\varphi)\leqno (5)$$ \vskip 1mm

$(df_1 +df_2{\bf j})(\varphi)=$

$$= \big(\frac{\partial f_1}{\partial z_1}dz_1 +\frac{\partial f_1}{\partial z_2}dz_2 - (\frac{\partial f_2}{\partial\overline z_1}d\overline z_1 -\frac{\partial f_2}{\partial\overline z_2}d\overline z_2){\bf j}\big)(\varphi_{11} dz_1\wedge d\overline z_1+\varphi_{12} dz_1\wedge d\overline z_2+\varphi_{21} dz_2\wedge d\overline z_1+\varphi_{22} dz_2\wedge d\overline z_2)$$  

$$= \frac{\partial f_1}{\partial z_1}dz_1\wedge(\varphi_{21} dz_2\wedge d\overline z_1+\varphi_{22} dz_2\wedge d\overline z_2) +\frac{\partial f_1}{\partial z_2}dz_2\wedge (\varphi_{11} dz_1\wedge d\overline z_1+\varphi_{12} dz_1\wedge d\overline z_2)$$

$$-\frac{\partial f_2}{\partial\overline z_1}d\overline z_1 {\bf j}\wedge(\varphi_{21} dz_2\wedge d\overline z_1+\varphi_{22} dz_2\wedge d\overline z_2) + \frac{\partial f_2}{\partial\overline z_2}d\overline z_2{\bf j}\wedge(\varphi_{11} dz_1\wedge d\overline z_1+\varphi_{12} dz_1\wedge d\overline z_2)$$

$$= (-\frac{\partial f_1}{\partial z_1}\varphi_{21} +\frac{\partial f_1}{\partial z_2}\varphi_{11})dz_1\wedge d\overline z_1 \wedge dz_2 +(\frac{\partial f_1}{\partial z_1}\varphi_{22} -\frac{\partial f_1}{\partial z_2}\varphi_{12})dz_1\wedge dz_2\wedge d\overline z_2 $$

$$+\big(\frac{\partial f_2}{\partial\overline z_1}\overline\varphi_{21}- \frac{\partial f_2}{\partial\overline z_2}\overline\varphi_{11} \big){\bf j} dz_1\wedge d\overline z_1 \wedge dz_2 +\big(-\frac{\partial f_2}{\partial\overline z_1}\overline\varphi_{22}-\frac{\partial f_2}{\partial\overline z_2}\overline\varphi_{12} \big){\bf j}dz_1\wedge dz_2\wedge d\overline z_2$$

$$=(A+B{\bf j})dz_1\wedge d\overline z_1 \wedge dz_2 +(C+D{\bf j})dz_1\wedge dz_2\wedge d\overline z_2$$

Take polar coordinates: $\lambda=\norm q\norm=(\vert z_1\vert ^2+\vert z_2\vert ^2)^{\frac{1}{2}}$ and the spherical coordinates on ${\bf S}^3$. Let $d\sigma$ the volume element on ${\bf S}^3$, and $K_j$, $(j=1,2)$ a convenient universal constant:

\centerline {$dz_1\wedge d\overline z_1\wedge dz_2\vert _{{\bf S}^3} =K_1 d\lambda\wedge d\sigma$}

\centerline {$dz_1\wedge dz_2 \wedge d\overline  z_2\vert _{{\bf S}^3}=K_2 d\lambda\wedge d\sigma$}\vskip 1mm

Same result for the part of any type of ${\rm Res}[\omega]$.

\subsubsection{\it Particular case} $f_1=\overline z_1$; $f_2=\overline z_2$. Then: ${\mathcal D}f=0$.
$$(d\overline z_1+d\overline z_2{\bf j})(\varphi)=\varphi_{22} d\overline z_1\wedge dz_2\wedge d\overline z_2+d\overline z_2{\bf j}\wedge\varphi_{12} dz_1\wedge d\overline z_2$$
$$=\varphi_{22} d\overline z_1\wedge dz_2\wedge d\overline z_2+\overline\varphi_{12} d\overline z_1\wedge d z_2\wedge d\overline z_2{\bf j}$$

$${\rm Res}[\omega](\varphi)=\lim_{\varepsilon\rightarrow 0}\int_{\vert
f\vert=\varepsilon}(\vert z_1\vert ^2+\vert z_2\vert ^2)^{-1}( z_1-\overline z_2{\bf j})(\varphi_{22} d\overline z_1\wedge dz_2\wedge d\overline z_2+\overline\varphi_{12} d\overline z_1\wedge d z_2\wedge d\overline z_2{\bf j})$$
$$=\lim_{\varepsilon\rightarrow 0}\int_{\norm
f\norm=\varepsilon}(\vert z_1\vert ^2+\vert z_2\vert ^2)^{-1}( z_1-\overline z_2{\bf j})(\varphi_{22}-\overline\varphi_{12}{\bf j}) d\overline z_1\wedge dz_2\wedge d\overline z_2)$$
$$=\lim_{\varepsilon\rightarrow 0}\int_{{\bf S}^3}(\vert z_1\vert ^2+\vert z_2\vert ^2)^{-1}( z_1-\overline z_2{\bf j})(\varphi_{22}-\overline\varphi_{12}{\bf j})\varepsilon d\sigma$$

$$={\rm res}[\omega](\varphi_{22}-\overline\varphi_{12}{\bf j})(0)=\delta_0(\varphi_{22}-\overline\varphi_{12}{\bf j})$$

$\displaystyle\int_{\norm f \norm=\varepsilon}$ means the integration on $\varepsilon{\bf S}^3$, where ${\bf S}^3$ is the unit 3-sphere and $d\sigma$ the volume element on ${\bf S}^3$, ${\rm res}[\omega]$, a constant playing the part of the Cauchy residue, and $\delta_0$ the Dirac measure at 0.

Same result for the part of any type of ${\rm Res}[\omega]$.\vskip 2mm

\subsubsection{$f$ holomorphic} If $f_2=0$ and $f=f_1$ is holomorphic, 

$${\rm Res}[\omega](\varphi)=\lim_{\varepsilon\rightarrow 0}\int_{\vert
f_1\vert=\varepsilon}\frac{df_1}{f_1}(\varphi)$$
Assume $f_1=z_1$, then:  $z_1=\varepsilon ^{i\theta}$, and
$${\rm Res}[\omega](\varphi)=\lim_{\varepsilon\rightarrow 0}\int_0^{2\pi}i d\theta \varphi_{2,2}(\varepsilon e ^{i\theta},z_2)dz_2\wedge d\overline z_2=2\pi i\lbrack Z(z_1)\rbrack(\varphi)$$
Same result for any holomorphic $f_1$ since $f_1=\varepsilon e^{i\theta}$.

\end{document}